\input xy
\xyoption{all}\xyoption{poly}
\newdir{ >}{{}*!/-9pt/@{>}}

\documentclass[reqno]{amsart}
\usepackage[T1]{fontenc}
\usepackage[utf8]{inputenc}

\usepackage{lineno,hyperref}
\modulolinenumbers[5]

\usepackage{tikz}
\usetikzlibrary{decorations.pathreplacing,matrix,positioning}
\usepackage{amsmath,latexsym,amssymb,mathrsfs,amsthm,mathtools} 

\usepackage{textcomp}
\usepackage{graphicx}

\usepackage{hyperref}
\usepackage[all]{xy}
\newcommand{\catC}{\mathcal{C}}

\newcommand{\Qnd}{\mathsf{Qnd}}

\newcommand{\Qndt}{\mathsf{Qnd}^{\ast}}

\newcommand{\Inn}{\mathsf{Inn}}
\newcommand{\Aut}{\mathsf{Aut}}
\newcommand{\Eq}{\mathsf{Eq}}

\newtheorem{theorem}{Theorem}[section]
\newtheorem{lemma}[theorem]{Lemma}
\newtheorem{proposition}[theorem]{Proposition}

\theoremstyle{definition}
\newtheorem{definition}[theorem]{Definition}
\newtheorem{examples}[theorem]{Examples}
\theoremstyle{remark}
\newtheorem{remark}[theorem]{Remark}
\usepackage{todonotes}
\usepackage{graphicx}
 \usepackage{color}

\newdir{ >}{{}*!/-10pt/\dir{>}}
\def\cartesien{%
    \ar@{-}[]+R+<6pt,-1pt>;[]+RD+<6pt,-6pt>%
    \ar@{-}[]+D+<1pt,-6pt>;[]+RD+<6pt,-6pt>%
  }

\begin{document}
	  \title{{Closure operators in the category of quandles}}
	  \author{Val\'erian Even}
	  %\footnote{valerian.even@uclouvain.be}
	  \author{Marino Gran}
	  	  \address{Institut de Recherche en Math\'ematique et Physique \\  Universit\'e catholique de Louvain \\ Chemin du Cyclotron 2, 1348 Louvain-la-Neuve, Belgium.}
	  %\footnote{marino.gran@uclouvain.be}
	  %\footnote{montoli@mat.uc.pt}
	  
\maketitle
%\begin{frontmatter}
%\title{Elsevier \LaTeX\ template\tnoteref{mytitlenote}}


%\cortext[cor2]{Corresponding author}

\begin{abstract}
	We study a regular closure operator in the category of quandles. We show that the regular closure operator and the pullback closure operator corresponding to the reflector from the category of quandles to its full subcategory of trivial quandles coincide, we give a simple description of this closure operator, and analyze some of its properties. The category of algebraically connected quandles turns out to be a connectedness in the sense of Arhangel'ski\v{\i} and Wiegandt corresponding to the full subcategory of trivial quandles, while the disconnectedness associated with it is shown to contain all quasi-trivial quandles. The separated objects for the pullback closure operator are precisely the trivial quandles.
	A simple formula describing the effective closure operator on congruences corresponding to the same reflector is also given.
\end{abstract}
%\begin{abstract}
%This template helps you to create a properly formatted \LaTeX\ manuscript.
%\end{abstract}

%\smallskip
%\noindent \textbf{Keywords.} Quandle, closure operator, connectedness, permutability of congruences

%\begin{keyword}
%\texttt{elsarticle.cls}\sep \LaTeX\sep Elsevier \sep template
%\MSC[2010] 00-01\sep  99-00
%\end{keyword}

%\end{frontmatter}

%\linenumbers

\section{Introduction}

A \emph{quandle} is a set $X$ equipped with two binary operations $\lhd$ and $\lhd^{-1}$ such that the following identities hold (for all $x,\ y,\ z \in X$):
\begin{itemize}
	\item[(A1)] $x \lhd x = x = x \lhd^{-1} x$ (idempotency);
	\item[(A2)] $(x \lhd y) \lhd^{-1} y = x = (x \lhd^{-1} y) \lhd y$ (right invertibility);
	\item[(A3)] $(x \lhd y) \lhd z = (x \lhd z) \lhd (y \lhd z)$ and $(x \lhd^{-1} y) \lhd^{-1} z = (x \lhd^{-1} z) \lhd^{-1} (y \lhd^{-1} z)$ (self-distributivity).
\end{itemize}

This structure, first studied by D.~Joyce~\cite{Joyce} and, independently, \linebreak by S.~V.~Matveev~\cite{M}, captures some fundamental properties of group conjugation: for example, the Wirtinger presentation of a knot group only involves relations of type $z = y^{-1} x y$ so that it is more natural to present a quandle rather than a group. In particular, the knot quandle of a knot is an invariant which is complete up to orientation.

If $X$ and $Y$ are two quandles, a function $f \colon X \rightarrow Y$ is a quandle homomorphism if it preserves the operations: $f(x \lhd x') = f(x) \lhd f(x')$ and $f(x \lhd^{-1} x') = f(x) \lhd^{-1} f(x')$ for all $x,\ x' \in X$. Quandles and quandle homomorphisms constitute a \emph{variety} of universal algebras. A quandle $X$ is \emph{trivial} when it satisfies the additional identity $x \lhd x' = x = x \lhd^{-1} x'$, for any $x,\ x' \in X$. By Birkhoff's Theorem the category $\Qndt$ of trivial quandles is a subvariety of the variety $\mathsf{Qnd}$ of all quandles, since it is obtained by requiring an additional identity to the ones defining the variety $\mathsf{Qnd}$ (see \cite{BS}, for instance). Observe that trivial quandles are the same thing as \emph{sets}: indeed, given a set $S$, the first projection $p_1 \colon S \times S \rightarrow S$ obviously yields the unique possible trivial quandle structure on $S$.

The forgetful functor $U \colon \Qndt \rightarrow \Qnd$ has a left adjoint 
$\pi_0 \colon \Qnd \rightarrow \Qndt$
\begin{equation}\label{adj}
\vcenter{\hbox{\begin{tikzpicture}[scale=0.3]
		\node[] (X) at (-5,0) {$\Qnd$};
		\node[] (Y) at (5,0) {$\Qndt$};
		\node[] (C) at (0,0) {$\perp$};
		\draw[->] (X) to [bend left=25] node[above]{$\pi_0$} (Y);
		\draw[<-] (X) to [bend right=25] node[below]{$U$} (Y);
		\end{tikzpicture}}\tag{A}}
\end{equation}
called the \emph{connected component functor}, which is constructed as follows.
The axioms in the definition of a quandle $X$ imply that the right actions, denoted by~$\rho_y \colon X \rightarrow X$ and defined by~$\rho_y(x) = x \lhd y$ for all $x \in X$, are automorphisms. Writing $\Inn(X)$ for the group of \emph{inner automorphisms} of $X$, i.e. the subgroup of the group $\Aut(X)$ of automorphisms of $X$ generated by all such $\rho_y$, one calls \emph{connected component} of $X$ an orbit under the action of $\Inn(X)$.
Two elements~$x$ and~$y$ of $X$ are in the same orbit if one can find a chain of elements~$x_i \in X$, for~$1~\leq~i~\leq~n$, linking the elements~$x$ and $y$
{ \[x \lhd^{\alpha_1} x_1 \lhd^{\alpha_2} \dots \lhd^{\alpha_n} x_n = y,\]} with $\lhd^{\alpha_i} \in \{\lhd, \ \lhd^{-1}\}$ for all $1 \leq i \leq n$, where one omits writing the parentheses according to the following convention:
\[x \lhd^{\alpha_1} x_1 \lhd^{\alpha_2} \dots \lhd^{\alpha_n} x_n := (\dots ((x \lhd^{\alpha_1} x_1) \lhd^{\alpha_2} x_2) \dots )\lhd^{\alpha_n} x_n.\] 
We write $[x]_X$ for the orbit of $x \in X$ under the action of $\Inn(X)$. 
The functor~$\pi_0 \colon \Qnd \rightarrow \Qndt$ sends a quandle $X$ to its trivial quandle (=set) $\pi_0(X)$ of connected components: $\pi_{0}(x)= [x]_X$.
The $X$-component $\eta_X \colon X \rightarrow U \pi_0 (X)$ of the unit of the adjunction is simply given by projection of $X$ to its trivial quandle (=set) of connected components.

In this article we investigate the pullback closure operator on subobjects~\cite{Holgate} in $\Qnd$ associated with the adjunction \eqref{adj}, which actually coincides with the corresponding regular closure operator (see \cite{S, DG, Clementino, DT}). We observe that the reflector $\pi_0 \colon \Qnd \to \Qndt$ preserves finite products (Lemma \ref{products}), and then study some basic properties of the pullback closure operator (Proposition \ref{properties}). It is then shown that the algebraically connected quandles are exactly the $c$-connected objects for this closure operator (Proposition \ref{connected}), the $c$-separated objects being the trivial quandles (Proposition \ref{separated}). The disconnectedness in the sense of Arhangel'ski\v{\i} and Wiegandt  \cite{AW} associated with the category of connected quandles is a category containing all the quasi-trivial quandles (in the sense of \cite{Inoue}), and is then strictly larger than the category of trivial quandles.

We then turn our attention to the study of the
\emph{effective closure operator} on congruences \cite{BGM}, again corresponding to the adjunction \eqref{adj}. In particular a simple formula in order to compute the closure is given (Proposition \ref{description}), which is based on a recent permutability result for quandles established in \cite{EG}. \\

\noindent {\bf Acknowledgements.} The authors would like to thank professors Maria Manuel Clementino and Sandra Mantovani for useful conversations on the subject of this article. The first author acknowledges the financial support of the F.N.R.S. (Cr\'edit pour bref s\'ejour \`a l'\'etranger) and of the project ``Fonds d'Appui \`a l'Internationalisation'' of the UCL that made his scientific visit  at the EPFL in Lausanne in June 2014 possible. It is during this research visit that part of this work was done. The first author also thanks the members of the Homotopy Theory Research Group at the EPFL for their kind hospitality.

\section{Pullback closure operator}\label{pullbackclosure}

In this article $\catC$ will always denote a finitely complete and \emph{regular} category (in the sense of \cite{Barr}): this means that
\begin{itemize}
	\item every arrow $f \colon X \rightarrow Y$ in $\mathcal C$ has a factorization $f = i \circ p$ 
	\begin{equation}\vcenter{\hbox{\begin{tikzpicture}[scale=0.7]
			\node[] (X) at (-2,0) {$X$};
			\node[] (Y) at (2,0) {$Y$};
			\node[] (I) at (0,-2) {$I$};
			\draw[->] (X) to node[above]{$f$} (Y);
			\draw[->] (X) to node[below left]{$p$} (I);
			\draw[->] (I) to node[below right]{$i$} (Y);
			\end{tikzpicture}\label{fact}}}
	\end{equation}
	with $p$ a regular epimorphism and $i$ a monomorphism ($I \xrightarrow{i} X$  is called the \emph{regular image} of $f$);
	\item regular epimorphisms are pullback stable in $\mathcal C$.
\end{itemize}
The classes $\mathcal E$ of regular epimorphisms and $\mathcal M$ of monomorphisms form a \emph{stable factorization system} $(\mathcal E,\mathcal M)$ in $\mathcal C$ \cite{CHK}. We shall refer to an arrow $M \xrightarrow{m} X$ in $\mathcal M$ as a \emph{subobject} of $X$, and write $m \in \mathsf{Sub}(X)$. Given any arrow $f \colon X \rightarrow Y$ we denote by $f(m)$ the subobject of $Y$ obtained by taking the regular image along~$f$ of $M \xrightarrow{m} X$, that is, the regular image of the composite $f \circ m$. When~$S \xrightarrow{s} Y$ is a subobject of $Y$ we write $f^{-1}(s)$ for the subobject of $X$ which is the \emph{inverse image} of $s$ along $f$.
\begin{definition} (Cf. \cite{DG, DT, CT})
	A \emph{closure operator} $c$ in $\catC$ associates, with any subobject $ M \xrightarrow{m} X$, another subobject $c_X(M) \xrightarrow{c_X(m)} X$, the closure of $m$ in $X$. This application satisfies the following properties for any $m,\ n \in \mathsf{Sub}(X)$, and $f : X \rightarrow Y$:
	\begin{enumerate}
		\item $m \le c_X(m)$;
		\item if $m \le n$, then $c_X(m) \le c_X(n)$;
		\item $f(c_X(m)) \le c_Y(f(m))$.
	\end{enumerate}
\end{definition}

We say that $ M \xrightarrow{m} X$ is \emph{closed} if $m = c_X(m)$, and that it is \emph{dense} \linebreak if $c_X(m) = 1_X$, for all $X \in \catC$. 
The closure operator factors every subobject~{$M \xrightarrow{m} X$} as \[M \xrightarrow{m/c_X(m)} c_X(M) \xrightarrow{c_X(m)}X,\]
where we write $m/c_X(m)$ for the unique arrow such that $c_X(m) \circ m/c_X(m) = m$.

The closure operator $c$ is said to be \emph{idempotent} if $c_X(m)$ is closed, and \emph{weakly hereditary} if $m/c_X(m)$ is dense.

Recall that a pointed endofunctor $(R,r)$ is given by an endofunctor $R \colon \catC \to~\catC$ and a natural transformation $r \colon 1_{\mathcal C} \to R$. Any pointed endofunctor in a category with a stable factorization system $(\mathcal E,\mathcal M)$ induces a corresponding closure operator, called the \emph{pullback closure operator} \cite{Holgate}, whose definition we are now going to recall.  If $M \xrightarrow{m} X$ belongs to $\mathcal{M}$, construct the following diagram where $n\circ e = Rm$ is the $(\mathcal{E},\mathcal{M})$-factorization of $Rm$, and $c_X(m)$ is the pullback of $n$ along $r_X$ :
%\begin{equation}\label{closurediagram}
%\vcenter{\xymatrix{M \ar[rr]^m \ar@{.>}[dr]_(.6){m/c_X(m)} \ar[dd]_{r_M}& & X \ar[dd]^{r_X}\\
%& c_X(M) \ar[dd] \ar[ur]_{c_X(m)}& \\
%RM \ar[dr]_e \ar@{-}[r] &\ar[r]^-{Rm} & RX \\
%& N \ar[ur]_n &}}
%\end{equation}

\begin{equation}
\vcenter{\hbox{\begin{tikzpicture}[scale=0.3]\label{closurediagram}
		\node[] (M) at (-8,7) {$M$};
		\node[] (X) at (8,7) {$X$};
		\node[] (N) at (0,-5) {$N$};
		\node[] (N') at (0,4) {$c_X(M)$};
		\node[] (P) at(0,-2) {};
		\node[] (M') at (-8,-2) {$RM$};
		\node[] (X') at (8,-2) {$RX$.};
		\draw[->] (M) to node[above]{$m$} (X);   
		\draw[->] (M) to node[left]{$r_M$} (M');
		\draw[->] (X) to node[right]{$r_X$} (X');
		\draw[->] (M') -- (P) to node[above]{$Rm$} (X');    
		\draw[->] (M') to node[below,rotate=-20]{$e$} (N);   
		\draw[->] (N) to node[below,rotate=20]{$i$} (X');
		\draw[dotted,->] (M) to node[below,rotate=-20]{$m/c_X(m)$} (N');
		\draw[->] (N') to node[below,rotate=20]{$c_X(m)$} (X);
		\draw[->] (N') to (N);  
		\draw[] (2,4) -- (2,2) -- (0.5,1.4);
		\end{tikzpicture}}}
\end{equation}

Then the assignment $c_X \colon \mathsf{Sub}(X) \rightarrow  \mathsf{Sub}(X)$ defined by $c_X(m) = r^{-1}_X(i)$ is a closure operator, called the pullback closure operator corresponding to $(R,r)$.

It is well-known that, in particular, any full reflective subcategory $\mathcal X$ of a regular category $\mathcal C$, with the property that each component of the unit of the adjunction is a regular epimorphism, induces an \emph{idempotent pullback closure operator} (see Corollary $7$ in \cite{Holgate} for instance). Let us describe it in the case of the reflection \eqref{adj} between the category $\mathsf{Qnd}$ of quandles and its full reflective subcategory $\mathsf{Qnd}^*$ of trivial quandles.

From now on, a subobject $M \xrightarrow{m} X$ in the category $\Qnd$ will be a subquandle inclusion. In particular the regular image $I$ of an arrow $f \colon X \to Y$ as in~\eqref{fact} will be the inclusion of the direct image $f(X)$ as a subquandle of $Y$.

\begin{lemma}\label{caracterisation}
	For a subobject $M \xrightarrow{m} X \in \Qnd$, \[c_X(M) = \{x \in X \ \vert\ x \in [a]_X \text{ for some } a \in M\}.\]
\end{lemma}

\begin{proof}
	The adjunction \eqref{adj} induces the pointed endofunctor $(U\pi_0, \eta)$, and the arrow $N \xrightarrow{n} RX= U\pi_0 X$ in the diagram \eqref{closurediagram} is simply the inclusion of the image of $U\pi_0 (m)$ in~$U\pi_0 (X)$. Since $c_X(m)$ is a monomorphism, it is the inclusion of the following subquandle of $X$:
	\[c_X(M) = \{x \in X \ \vert\ x \in [a]_X \text{ for some } a \in M\}\]
\end{proof}

The pullback closure operator takes a subquandle $M$ of $X$ and extends it to the union of connected components that are ``touched'' by $M$. For instance, one can represent the action of the closure operator on a subobject $M$ of $X$ (represented by the dark grey rectangle here below) by the light grey part $c_X(M)$ in the following diagram:

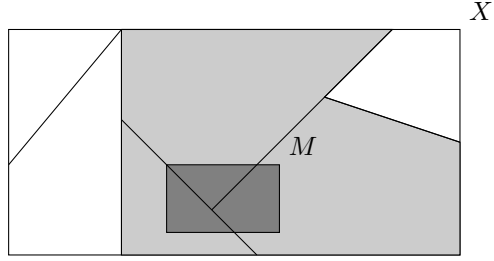
\[\begin{tikzpicture}[scale=0.3]
\node[above right] (X) at (10,5) {$X$};
\draw[] (-10,-5) rectangle (10,5);
\draw[fill=white!80!black] (-5,5) -- (-5,-5) -- (10,-5) -- (10,0) -- (4,2) -- (7,5) -- (-5,5);
\draw[fill=gray] (-3,-4) rectangle (2,-1) node[above right]{$M$};
\draw[] (-10,-1) -- (-5,5);
\draw[] (-5,5) -- (7,5);
\draw[] (-5,5) -- (-5,-5);
\draw[] (-5,1) -- (1,-5);
\draw[] (-1,-3) -- (7,5);
\draw[] (4,2) -- (10,0);
\end{tikzpicture}\]

\begin{remark}\label{regular=pullback}
	Given a reflective subcategory $\mathcal{A}$ of a regular category $\catC$ with reflector $I$ and unit $r$, one can also consider the classical closure operator for subobjects in $\catC$, called the 
	\emph{regular closure operator} \cite{S, DG, Clementino, CT}. This associates, with any 
	subobject $M \xrightarrow{m} X$ in $\catC$, the equalizer $(c_X^{reg}(M), c_X^{reg}(m))$ of~{$r_{X +_M X} \circ i$} and $r_{X +_M X} \circ j$, where $(i,\ j : X \to X +_M X)$ is the cokernel pair of $m$:
	
	\[\begin{tikzpicture}[scale=1.2]
	\node (M) at (-3,0) {$M$};
	\node (X) at (-1,0) {$X$};
	\node (S) at (1,0) {$X +_{M} X$};
	\node (R) at (4,0) {$UI(X +_M X)$.};
	\node (C) at (-2,-1) {$c_X^{reg} (M)$};
	\draw[dotted,->] (M) to (C);
	\draw[->] (C) to node[below right]{$c_X^{reg} (m)$} (X);
	\draw[->] (M) to node[above]{$m$} (X);
	\draw[transform canvas={yshift=0.5ex},->] (X) to node[above]{$i$} (S);
	\draw[transform canvas={yshift=-0.5ex},->] (X) to node[below]{$j$} (S);
	\draw[->] (S) to node[above]{$r_{X +_M X}$} (R);
	\end{tikzpicture}\]
	
	For the reflection \eqref{adj} between the category of quandles and the category of trivial quandles, one observes that any monomorphism in the subcategory $\Qndt$ is a \emph{regular monomorphism} (since $\Qndt \cong \mathsf{Set}$). As it is well known (and easy to check) the regular closure operator and the pullback closure operator induced by this reflection then coincide : $c_X = c_X^{reg},$ for every $X \in \Qnd$.
\end{remark}

\begin{lemma}\label{products}
	The functor $\pi_0 \colon \Qnd \rightarrow \Qndt$ preserves finite products.
\end{lemma}

\begin{proof}
	It suffices to check that the functor $\pi_0$ preserve binary products, since it preserves the terminal object. Let $X,\ Y \in \Qnd$. There is a unique quandle homomorphism $\gamma \colon \pi_0(X \times Y) \to \pi_0(X) \times \pi_0(Y)$ such that $\gamma ([(x,y)]_{X\times Y}) = ([x]_{X},[y]_Y)$. It is easy to see that $\gamma$ is surjective, by using the fact that each component of the unit of the adjunction \eqref{adj} is surjective.
	
	Let us check that $\gamma$ is injective : let $[(x,y)]_{X\times Y}$ and $[(x',y')]_{X\times Y}$ be elements of $\pi_0({ X \times Y})$ such that $\gamma ([(x,y)]_{X\times Y}) = \gamma ([(x',y')]_{X\times Y})$. This means that { $[x]_{X}=[x']_{X}$ and $[y]_{Y}=[y']_{Y}$}. There are then $x_{i} \in X$, $y_j \in Y$ and $\lhd^{\alpha_{i}},\ \lhd^{\beta_j} \in \{\lhd, \lhd^{-1}\}$ for $1 \leq i \leq n$ and $1 \leq j \leq m$ such that $$x \lhd^{\alpha_{1}} x_{1} \dots \lhd^{\alpha_{n}} x_{n} = x'$$ and $$y \lhd^{\beta_{1}} y_{1} \dots \lhd^{\beta_{m}} y_{m} = y'.$$ The idea now is to use the idempotency of $\lhd$ and $\lhd^{-1}$ in order to pass from $x$ to $x'$, and then from $y$ to $y'$, without changing the other component~:
	{ \[
		(x,y) \lhd^{\alpha_{1}} (x_1,y) \lhd^{\alpha_2} \dots \lhd^{\alpha_{n}} (x_n,y) \lhd^{\beta_{1}} (x',y_1) \dots \lhd^{\beta_{m}} (x',y_{m})
		= (x',y').
		\]}
	This shows that $[(x,y)]_{X \times Y} = [(x',y')]_{X \times Y}$, proving that $\gamma$ is also injective, thus an isomorphism.
\end{proof}

\begin{lemma}\label{sup}
	Let $S \xrightarrow{s} X$ and $T \xrightarrow{t} X$ be two subquandles of $X \in \Qnd$. The smallest subquandle $S \vee T$ containing both $S$ and $T$ is given by the set \[U = \left\{a_1 \lhd^{\alpha_1} a_2 \lhd^{\alpha_2} \cdots \lhd^{\alpha_{n-1}} a_n \ \vert \ a_i \in S \cup T \text{ and } \lhd^{\alpha_i} \in \{\lhd, \lhd^{-1}\}\ \forall 1 \leq i \leq n-1 \right\}\] equipped with the quandle operation inherited from $X$.
\end{lemma}

\begin{proof}
	First note that in a quandle $X$ we have the following equality for all $x,\ y,\ z \in X$ and $\lhd^\alpha, \ \lhd^{\beta} \in \{\lhd, \lhd^{-1}\}$ :
	\begin{align*}
	x \lhd^\alpha \left(y \lhd^\beta z\right) &= \left(\left( x\lhd^{-\beta} z \right)\lhd^\beta z\right) \lhd^\alpha \left(y \lhd^\beta z\right) \\
	&= \left( \left(x \lhd^{-\beta} z\right) \lhd^\alpha y \right) \lhd^\beta z
	\end{align*}
	
	Given two elements $a_1 \lhd^{\alpha_1} a_2 \lhd^{\alpha_2} \cdots \lhd^{\alpha_{n-1}} a_n$ and $b_1 \lhd^{\beta_1} b_2 \lhd^{\beta_2} \cdots \lhd^{\beta_{m-1}} b_m$ of  $U$, the previous equality gives 
	\begin{align*}
	&\left( a_1 \lhd^{\alpha_1} a_2 \lhd^{\alpha_2} \cdots \lhd^{\alpha_{n-1}} a_n \right) \lhd^{\alpha} \left( b_1 \lhd^{\beta_1} b_2 \lhd^{\beta_2} \cdots \lhd^{\beta_{m-1}} b_m \right) = \\
	&a_1 \lhd^{\alpha_1} a_2 \lhd^{\alpha_2} \cdots \lhd^{\alpha_{n-1}} a_n \lhd^{-\beta_{m-1}} b_m \lhd^{-\beta_{m-2}} \cdots \lhd^{-\beta_1} b_2  \lhd^{\alpha} b_1 \lhd^{\beta_1} b_2 \lhd^{\beta_2} \cdots \lhd^{\beta_{m-1}} b_m
	\end{align*}
	showing that $U$ is stable under $\lhd$ and $\lhd^{-1}$.
	
	Certainly $U$ contains both $S$ and $T$, and any quandle containing $S$ and $T$ must contain all chains of the form $a_1 \lhd^{\alpha_1} a_2 \lhd^{\alpha_2} \cdots \lhd^{\alpha_{n-1}} a_n$ with $a_i \in S \cup T$ so that $U$ is the smallest subquandle containing $S$ ans $T$.
\end{proof}

Let us now show some basic properties of the pullback closure operator associated with \eqref{adj}:

\begin{proposition}\label{properties}
	The pullback closure operator $c$ for the adjunction between $\Qnd$ and $\Qndt$ has the following properties :
	\begin{itemize}
		\item[(1)] $c_X(\{x\}) = [x]_X$ for any $x \in X$;
		\item[(2)] $c_X(\bigvee_{i \in I}s_i) = \bigvee_{i \in I} c_X(s_i)$ for subobjects $S_i \xrightarrow{s_i} X$ with $i \in I$ ($c$ is fully additive);
		\item[(3)] $c_{X}(\prod_{1 \leq i \leq n} m_i) = \prod_{1 \leq i \leq n} c_{X_i} (m_i)$, where $X = \prod_{1 \leq i \leq n}X_i $ for any finite family of subobjects $M_i \xrightarrow{m_i} X_i$ with $1 \leq i \leq n$ ($c$ is finitely productive);
		\item[(4)] $f(c_X(m)) = c_Y(f(m))$ for any surjective homomorphism $f \colon X \to Y$ and subobject $M \xrightarrow{m} X$.
	\end{itemize}
\end{proposition}

\begin{proof}
	The first point is easily verified.
	To check $(2)$, observe that Lemma~\ref{sup} says that $$\bigvee_{i \in I}S_i = \{a_1 \lhd^{\alpha_1} a_2 \lhd^{\alpha_2} \dots \lhd^{\alpha_{n-1}} a_n \in X \ \vert \ a_j \in \bigcup_{i \in I} S_i \text{ for all } 1\leq j \leq n\},$$ where the quandle structure is inherited from $X$. Accordingly: \begin{align*}c_X(\bigvee_{i \in I}S_i) &= \{x \in X\ \vert \ x \in [a]_X \text{ with } a \in \bigvee_{i \in I}S_i  \} \\
	&= \{x \in X\ \vert \ x \in [a]_X \text{ with } a \in \bigcup_{i \in I}S_i \} \\
	&= \bigvee_{i \in I} \{ x \in X\ \vert \ x \in [a]_X \text{ with } a \in S_i\}\\
	&= \bigvee_{i \in I} c_X(S_i).
	\end{align*}
	
	In order to check $(3)$, let us write $(a_i)$ for an element of $\prod_{1 \leq i \leq n} M_i = M$ and $(x_i)$ for an element of $\prod_{1 \leq i \leq n} X_i = X$. Then 
	\begin{align*}
	c_X (M) &= \{(x_i) \in X \ \vert\ (x_i) \in [(a_i)]_X \text{ for some } (a_i) \in M\} \\
	&= \{(x_i) \in X \ \vert\ x_i \in [a_i]_{X_i} \text{ for some } a_i \in M_i \text{ for all } 1 \leq i \leq n\} \text{ (by Lemma~\ref{products})} \\
	&= \prod_{1 \leq i \leq n} \{x_i \in X_i \ \vert\ x_i \in [a_i]_{X_i} \text{ for some } a_i \in M_i\} \\
	&= \prod_{1 \leq i \leq n} c_{X_i} (M_i).
	\end{align*} 
	
	To see $(4)$, we have to check the validity of $c_Y(f(m)) \subset f(c_X(m)) $ for a surjective quandle homomorphism $f \colon X \to Y$. If $f(x) \in  c_Y(f(m))$ then there exist $y_i \in Y$ and $\lhd^{\alpha_i} \in \{\lhd, \lhd^{-1}\}$ for $1\leq i \leq n$ such that for some $a \in M$ $$f(x) = f(a) \lhd^{\alpha_1} y_1 \dots \lhd^{\alpha_n} y_n.$$ But since $f$ is surjective, there exists $x_i \in X$ such that $f(x_i) = y_i$ for all $1 \leq i \leq n$ so that {\begin{align*}
		f(x) &= f(a) \lhd^{\alpha_1} y_1 \lhd^{\alpha_2}\dots \lhd^{\alpha_n} y_n \\
		&= f(a) \lhd^{\alpha_1} f(x_1) \lhd^{\alpha_2}\dots \lhd^{\alpha_n} f(x_n) \\
		&= f(a \lhd^{\alpha_1} x_1 \lhd^{\alpha_2}\dots \lhd^{\alpha_n} x_n) \\
		&\in f(c_X(m)).
		\end{align*}}
\end{proof}

\begin{remark} We now show that the pullback closure operator associated with the reflection \eqref{adj} is not weakly hereditary. Consider the $3$-element quandle $X$ with $\lhd = \lhd^{-1}$, and $\lhd$ defined by the following table :
	\begin{equation}\label{monster}
	\begin{tabular}{|c|c|c|c|}
	\hline
	$\lhd$ & $x$ & $y$ & $z$ \\
	\hline
	$x$ & $x \lhd x = x$ & $x \lhd y = x$ & $x \lhd z = y$ \\
	\hline
	$y$ & $y \lhd x = y$ & $y \lhd y = y$ & $y \lhd z = x$ \\
	\hline
	$z$ & $z \lhd x = z$ & $z \lhd y = z$ & $z \lhd z = z$ \\
	\hline
	\end{tabular}
	\end{equation}  Now, if we look at the closure of the subobject $\{x\} \xrightarrow{m} X$, where $\{x \}$ is a one-element quandle, we find that $c_X (\{x\}) = \{x, y\}$ equipped with trivial quandle operations. Thus $m/c_X(m)$ is defined as $m/c_X(m) (x) = x$, and $c_{c_X(\{x\})}(\{x \}) = c_{\{x, y\}}(\{x \}) =\{x\}$, which is not isomorphic to $c_X (\{x\})$.
\end{remark}

\subsection{Connected quandles}
In a category $\catC$ equipped with a closure operator $c$ one says that {an object}~$X$ is \emph{$c$-connected} if the diagonal $\Delta_X \colon X \to X \times X$ is dense. 
In the category $\Qnd$ a quandle $A$ is called \emph{algebraically connected} if $A$ has exactly one orbit \cite{Joyce}. We now show that the $c$-connected quandles for the pullback closure operator associated with \eqref{adj} are precisely the algebraically connected quandles:
\begin{proposition}\label{connected}
	Let $c$ be the pullback closure operator for the adjunction \eqref{adj}. A quandle $X$ is $c$-connected if and only if it is algebraically connected.
\end{proposition}
\begin{proof}
	When $X$ is algebraically connected, so that $\pi_0(X) = \{\star\}$ is the one-element quandle, then $X \times X$ is also algebraically connected, by Lemma \ref{products}. 
	We then have the commutative diagram 
	\[\begin{tikzpicture}[scale=0.35]
	\node[] (M) at (-8,7) {$X$};
	\node[] (X) at (8,7) {$X \times X$};
	\node[] (N) at (0,-5) {$N$};
	\node[] (N') at (0,4) {$c_{X \times X}(X)$};
	\node[] (P) at (0,-2) {};
	\node[] (M') at (-8,-2) {$U\pi_0(X)$};
	\node[] (X') at (8,-2) {$U\pi_0(X \times X)$};
	\draw[->] (M) to node[above]{$\Delta_X$} (X);   
	\draw[->] (M) to node[left]{$\eta_X$} (M');
	\draw[->] (X) to node[right]{$\eta_{X \times X}$} (X');
	\draw[->] (M') -- (P) to node[above]{$U\pi_0(\Delta_X)$} (X');    
	\draw[->] (M') to node[below,rotate=-20]{$e$} (N);   
	\draw[->] (N) to node[below,rotate=20]{$i$} (X');
	\draw[dotted,->] (M) to (N');
	\draw[->] (N') to node[below,rotate=20]{$c_{X \times X}(\Delta_X)$} (X);
	\draw[->] (N') to (N);  
	\end{tikzpicture}\]
	where $i$ is an isomorphism. Accordingly, $c_{X \times X}(\Delta_X)$ is an isomorphism, and $X$ is $c$-connected.
	
	Conversely, assume now that $X$ is $c$-connected. Since $c_{X \times X}(\Delta_X)$ is an isomorphism and $\eta_{X \times X}$ is a regular epimorphism, the monomorphism $i$ is an isomorphism. It follows that $U\pi_0(\Delta_X)$, and then $\pi_0(\Delta_X)$, is surjective: this means that for any $([x],[y]) \in \pi_0(X \times X)= \pi_0(X) \times \pi_0 (X)$, there exists a $[z] \in \pi_0(X)$ such that $[x]=[z]=[y]$, showing that $X$ is algebraically connected.
\end{proof}
From now on we shall call a quandle \emph{connected} when it satisfies the equivalent conditions in Proposition \ref{connected}.

A similar result holds for the so-called \emph{$c$-separated objects}: these turn out to be exactly the so-called trivial quandles. Recall that an object is said to be $c$-separated for a closure operator $c$ when $\Delta_X \colon X \to X \times X$ is closed.
\begin{proposition}\label{separated}
	Let $c$ be the pullback closure operator for the adjunction \eqref{adj}. A quandle $X$ is $c$-separated if and only if it is a trivial quandle.
\end{proposition}
\begin{proof}
	By taking into account Lemma \ref{products} we see that a quandle $X$ is $c$-separated if and only if the commutative square 
	\[\begin{tikzpicture}[scale=0.35]
	\node[] (M) at (-4,4) {$X$};
	\node[] (X) at (4,4) {$X \times X$};
	\node[] (M') at (-4,-4) {$U\pi_0(X)$};
	\node[] (X') at (4,-4) {$U\pi_0(X \times X)$};
	\draw[->] (M) to node[above]{$\Delta_X$} (X);   
	\draw[->] (M) to node[left]{$\eta_X$} (M');
	\draw[->] (X) to node[right]{$\eta_{X \times X}$} (X');
	\draw[->] (M') to node[below]{$\Delta_{U\pi_0(X)}$} (X');    
	\end{tikzpicture}\]
	is a pullback.
	Since this square is a pullback if and only if the kernel pair $\Eq(\eta_X)$ of the unit $\eta_X$ is the discrete equivalence relation on $X$, this is also equivalent to the fact that $\eta_X$ is a monomorphism. But $\eta_X$ is always a regular epimorphism, so that $X$ is $c$-separated if and only if $\eta_X$ is an isomorphism, as desired.
	%If a quandle $X$ is trivial, it is easy to see that  is closed. Conversely, when the diagonal is closed, 
	%the canonical map $\Delta_X/c_{X \times X}(\Delta_X)$ sending $x \in X$ to $(x,x) \in c_{X \times X} (X)$ is bijective. Given two elements %$y$ and $z$ in $X$ such that $[y]_X=[z]_X$, so that $(y,z)\in c_{X \times X} (X)$, the assumption says that there is a unique $x \in X$ such %that $(\Delta_X/c_{X \times X}(\Delta_X))(x)=(x,x) =(y,z)$, thus $y=z$.
\end{proof}
We are now going to show that the connected quandles form a \emph{connectedness} in the sense of Preuss, Herrlich, Arhangel'ski\v{\i} and Wiegandt \cite{AW} with respect to the class of trivial quandles. We also give a description of the disconnectedness associated with the class of connected quandles. 
We will follow~\cite{DT2} and define a morphism~$f \colon X \to Y$ to be \emph{constant} if $!_X \colon X \to 1$ (where $1$ is the terminal object) is a strong epimorphism and a factor of $f$. In the category $\Qnd$ of quandles, this means that $f \colon X \to Y$ is constant if and only if it factors through the one-element quandle $\{ \star \}$.

For a full subcategory $\mathcal{X}$ of $\catC$, the class \[r(\mathcal{X}) \coloneqq \{C \in \catC \ \vert\ \text{every } f\colon  X \rightarrow C \text{ is constant for all } X\in \mathcal{X}\}\]
is called a \emph{disconnectedness}, and \[l(\mathcal{X}) \coloneqq \{C \in \catC \ \vert\ \text{every } f\colon C \rightarrow X \text{ is constant for all } X\in \mathcal{X}\}\] is called a \emph{connectedness}.
There is an adjunction, with $ \mathcal{S}ub(\catC)$ the class of all full subcategories of $\catC$ ordered by inclusion:
\begin{equation} 
\vcenter{\hbox{\begin{tikzpicture}[scale=0.4]\label{adjsub}
		\node[] (X) at (-5,0) { $\mathcal{S}ub(\catC)$};
		\node[] (Y) at (5,0) {$\mathcal{S}ub(\catC)^{op}$};
		\node[line width=4pt] (C) at (0,0) {$\perp$};
		\draw[->,transform canvas={yshift=1.3ex}] (X) to node[above]{$r$} (Y);
		\draw[<-,transform canvas={yshift=-1.3ex}] (X) to node[below]{$l$} (Y);
		\end{tikzpicture}}}
\end{equation}

In the category $\mathsf{Top}$ of topological spaces, we have \begin{align*}
\mathcal{Y} & \coloneqq \{ \text{connected spaces}\} = l(r(\mathcal{Y})) \\
\mathcal{Z} & \coloneqq \{\text{hereditarily disconnected spaces}\} = r(\mathcal{Y})
\end{align*}
We are going to show that there is a similar correspondence in the category $\Qnd$ of quandles. In the following proposition by \emph{trivial subquandle} we shall mean the empty subquandle, and any one-element subquandle of a given quandle.
\begin{theorem}
	In the category $\Qnd$, given $\mathcal{X} = \Qndt $ we have
	\[\mathcal{Y} \coloneqq \{\text{connected quandles}\} = l(\mathcal{X}) = l(r(\mathcal{Y}))  \]
	and
	\[\mathcal{Z} \coloneqq\{X \in \Qnd\ \vert\ X \text{ has no non-trivial connected subquandles}\} = r(\mathcal{Y}). \]
\end{theorem}

\begin{proof}
	\begin{enumerate}
		\item $\mathcal{Y} = l(\mathcal{X})$
		
		If $X$ is connected, then any $f \colon X \rightarrow Y$ with $Y \in \Qndt$ is constant by commutativity of the following square 
		\[\begin{tikzpicture}[scale=0.27]
		\node[] (M) at (-8,7) {$X$};
		\node[] (X) at (8,7) {$Y$};
		\node[] (M') at (-8,-2) {$U\pi_0(X) = \{\star\}$};
		\node[] (X') at (8,-2) {$Y=U\pi_0(Y)$};
		\draw[->] (M) to node[above]{$f$} (X);   
		\draw[->] (M) to node[left]{$\eta_X$} (M');
		\draw[double distance = 1.5pt] (X) to node[right]{$\eta_Y$} (X');
		\draw[->] (M') to node[below]{$\pi_0(f)$} (X');    
		\end{tikzpicture}\]
		
		To see that $l(\mathcal{X}) \subset \mathcal{Y}$, suppose that every $f \colon X \rightarrow Y$ is constant for \linebreak all {$Y \in \mathcal{X}$}, and let us prove that $\pi_0(X) = \{\star\}$. Note that {$\eta_X \colon X \to ~U~\pi_0(X)$} is constant, but it is a regular epimorphism, thus~$\pi_0(X) = \{\star\}$.
		%Consider the natural diagram 
		%\[\xymatrix@=1pc{X  \ar[rr]^f \ar[dd]_{\eta_X} \ar[dr]& & Y \ar@{=}[dd] \\
		%& \{\star\} \ar[ur]& \\
		%\pi_0(X) \ar[rr]_{\pi_0(f)} & & Y}\]
		%Since $\eta_X$ is a strong epimorphism and $\{\star\} \rightarrow Y$ is a monomorphism, there is an induced arrow $\alpha \colon \pi_0(X) \rightarrow \{\star\}$
		%\[\xymatrix@=1pc{X  \ar[rr]^f \ar[dd]_{\eta_X} \ar[dr]& & Y \ar@{=}[dd] \\
		%& \{\star\} \ar[ur]& \\
		%\pi_0(X) \ar[ur]^{\alpha} \ar[rr]_{\pi_0(f)} & & Y}\]
		%But $X \rightarrow \pi_0(X)$ is also constant (since $\pi_0(X)$ is trivial), thus
		%\[\xymatrix@=1pc{X  \ar[rr]^f \ar[dd]_{\eta_X} \ar[dr]& & Y \ar@{=}[dd] \\
		%& \{\star\} \ar[ur] \ar@<2pt>[dl]^{\beta} \ar@{<-}@<-2pt>[dl]_{\alpha} & \\
		%\pi_0(X)  \ar[rr]_{\pi_0(f)} & & Y}\]
		%Then $\pi_0(X)$ is isomorphic to $\{\star\}$, so $X$ is a connected quandle.
		
		\item $\mathcal{Z} = r(\mathcal{Y})$
		
		First take $X \in \mathcal{Z}$, $Y$ a connected quandle, and $f \colon Y \to X$. By taking the regular epimorphism-monomorphism factorization $i \circ p$ of $f$
		\[\begin{tikzpicture}[scale=0.7]
		\node[] (X) at (-2,0) {$Y$};
		\node[] (Y) at (2,0) {$X$};
		\node[] (I) at (0,-2) {$f(Y)$};
		\draw[->] (X) to node[above]{$f$} (Y);
		\draw[->] (X) to node[below left]{$p$} (I);
		\draw[->] (I) to node[below right]{$i$} (Y);
		\end{tikzpicture}\]
		observe that $f(Y)$ is connected as a quotient of a connected quandles, but it is also a subquandle of $X$ so it must be trivial, thus $f(Y) = \{ \star \}$. 
		
		Now suppose $X \in r(\{\text{connected quandles}\})$, and that $$X \notin \mathcal{Z} = \{X \in \Qnd\ \vert\ X \text{ has no non-trivial connected subquandles}\}.$$ Then $X$ has a non-trivial connected subquandle $\Gamma$ of cardinality strictly greater than $2$ (the only $2$-element quandle is trivial), with inclusion~{$\Gamma~\to~X$.} But $X \in r(\{\text{connected quandles}\})$ so the inclusion $\Gamma \to X$ factors through~$\{ \star \}$, thus $\Gamma = \{ \star \}$, a contradiction.
		
		\item $\mathcal{Y} = l(\mathcal{Z})$
		
		This follows from the adjunction~\eqref{adjsub}.
		%We only need to prove that $l(\mathcal{Z})  \subset \mathcal{Y}$.
		%If $X$ is not connected, this means that $X$ has at least two connected components, one of which will be written $X_1$. We define a quandle homomorphism $f \colon X \to Y$, where $Y = \{y,z\}$ is the trivial quandle of cardinality $2$, by setting $f(X_1) = y$ and $f(X \setminus X_1) = z$. This quandle homomorphism is not constant and $Y \in \mathcal{Z}$ so $X \notin l(\mathcal{Z})$. 
	\end{enumerate}
	
\end{proof}

\begin{remark}
	One might wonder whether $\mathcal{Z} = \Qndt$. Certainly the class of trivial quandles is contained in $\mathcal{Z}$ but the converse is not true because the larger class of quasi-trivial quandles is also contained in $\mathcal{Z}$. A \emph{quasi-trivial} quandle~\cite{Inoue} is a quandle satisfying { $x \lhd (x \lhd^{\alpha_1} x_1 \lhd^{\alpha_2} \cdots \lhd^{\alpha_n} x_n) = x$} for all $x,\ x_i \in X$ and~$\lhd^{\alpha_i} \in \{\lhd, \lhd^{-1}\}$ with $1 \leq i \leq n$. 
	
	To prove that any quasi-trivial quandle $X$ is in $\mathcal{Z}$, let $i \colon Y \to X$ be an inclusion of a connected subquandle $Y$ of $X$.
	Then $Y$ is also a quasi-trivial quandle since $i$ is injective and {$$i(x \lhd (x \lhd^{\alpha_1} x_1 \lhd^{\alpha_2}\cdots \lhd^{\alpha_n} x_n)) =   i(x) \lhd (i(x) \lhd^{\alpha_1} i(x_1) \lhd^{\alpha_2}\cdots \lhd^{\alpha_n} i(x_n)) = i(x).$$}
	Since $Y$ is connected, for any $x$, $y$ in $Y$ there exist $x_i \in Y$ and $\lhd^{\alpha_i} \in \{\lhd, \lhd^{-1}\}$  for $1 \leq i \leq n$  such that { $x \lhd^{\alpha_1} x_1 \lhd^{\alpha_2}\cdots \lhd^{\alpha_n} x_n =y$.} The fact that $Y$ is quasi-trivial gives { $$x \lhd y = x \lhd (x \lhd^{\alpha_1} x_1\lhd^{\alpha_2} \cdots \lhd^{\alpha_n} x_n) = x.$$}
	The quandle $Y$ is then trivial, and then it belongs to $\mathcal{Z}$.
	
	Note that a quasi-trivial quandle is not trivial in general: an example is provided by the quandle~\eqref{monster}.
\end{remark}

\section{Effective closure operators}

In this section we recall some results about closure operators on \emph{effective equivalence relations} in regular categories, which can be found in~\cite{BGM}. We then describe the closure operator corresponding to the reflection $\eqref{adj}$ between quandles and trivial quandles.

An internal equivalence relation $\xymatrix{R  \ar@<2pt>[r]^{r_1}   \ar@<-2pt>[r]_{r_2} & X}$ on an object $X$ in $\mathcal C$ (see \cite{Borceux}, Section $2.5$, for instance), where $r_1$ and $r_2$ denote the projections, is \emph{effective}~\cite{Barr} if there is an arrow $g \colon X \rightarrow Y$ such that $(R, r_1, r_2)$ is the \emph{kernel pair} of $g$, i.e. the following square is a pullback:
\[\begin{tikzpicture}[scale=0.2]
\node[] (M) at (-4,4) {$R$};
\node[] (X) at (4,4) {$X$};
\node[] (M') at (-4,-4) {$X$};
\node[] (X') at (4,-4) {$Y$};
\draw[->] (M) to node[above]{$r_2$} (X);   
\draw[->] (M) to node[left]{$r_1$} (M');
\draw[->] (X) to node[right]{$g$} (X');
\draw[->] (M') to node[below]{$g$} (X');    
\end{tikzpicture}\]
We denote such an equivalence relation simply by $R$. If $f \colon Y \rightarrow X$ is an arrow in~$\mathcal C$, we write $f^{-1}(R)$ for the equivalence relation on $Y$ which is the inverse image of $R$ along $f$: it is obtained by composing the dotted arrow in the following pullback with the projections $\xymatrix{Y \times Y  \ar@<2pt>[r]^-{p_1}   \ar@<-2pt>[r]_-{p_2} & Y}$:
\[\begin{tikzpicture}[scale=0.3]
\node[] (M) at (-4,4) {$f^{-1}(R)$};
\node[] (X) at (4,4) {$R$};
\node[] (M') at (-4,-4) {$Y \times Y$};
\node[] (X') at (4,-4) {$X \times X$};
\draw[->] (M) to (X);   
\draw[->,dotted] (M) to (M');
\draw[->] (X) to node[right]{$(r_1,r_2)$} (X');
\draw[->] (M') to node[below]{$f\times f$} (X');    
\end{tikzpicture}\]

\begin{definition} \cite{BGM}
	An \emph{effective closure operator} $c$ on effective equivalence relations in a regular category $\mathcal C$ consists in giving, for every effective equivalence relation $R$ on an object $X$, another effective equivalence $c_X(R)$, called the \emph{closure} of $R$. This assignment has to satisfy the following properties, where $R$ and $S$ are effective equivalence relations on $X$, $f \colon Y \rightarrow X$ is a morphism in $\mathcal{C}$:
	\begin{itemize}
		\item[(1)] $R \le c_X(R)$;
		\item[(2)] $R \le S$ implies $c_X(R) \le c_X(S)$;
		\item[(3)] $c_Y({f^{-1}(R)}) \le f^{-1}(c_X({R}))$;
		\item[(4)] $c_X(c_X({R})) = c_X({R})$.
		\item[(5)] if $f \colon Y \rightarrow X$ is a \emph{regular epimorphism}, one then has the equality $$c_Y({f^{-1}(R)}) = f^{-1}(c_X({R})).$$
	\end{itemize}
\end{definition}
By an \emph{$\mathcal E$-reflective subcategory $\mathcal{X}$} of $\mathcal{C}$ we shall mean a full reflective subcategory $\mathcal{X}$ of a regular category $\mathcal{C}$ 

\begin{equation}\label{generaladj}
\vcenter{\hbox{\begin{tikzpicture}[scale=0.3]
		\node[] (X) at (-5,0) {$\mathcal C$};
		\node[] (Y) at (5,0) {$\mathcal X$};
		\node[] (C) at (0,0) {$\perp$};
		\draw[->] (X) to [bend left=25] node[above]{$F$} (Y);
		\draw[<-] (X) to [bend right=25] node[below]{$U$} (Y);
		\end{tikzpicture}}\tag{B}}
\end{equation}
with the property 
that each component $\eta_X \colon X \rightarrow UF (X)$ of the unit $\eta$ of the adjunction is in $\mathcal E$ (i.e. a regular epimorphism).
$\mathcal E$-reflective subcategories of a regular category can be characterized in terms of effective closure operators as follows:

\begin{theorem}\cite{BGM}
	Let $\mathcal{C}$ be a regular category. There is a bijection between the $\mathcal E$-reflective subcategories of $\mathcal{C}$ and the effective closure operators in $\mathcal C$.
\end{theorem}
The existence of this bijection was proved in Theorem 2.3 in~\cite{BGM}. In order to make the article more self-contained, we recall how the closure $c_X (R)$ of an effective equivalence relation $R$ on $X$ is defined starting from an $\mathcal E$-reflective subcategory $\mathcal X$ of a regular category $\mathcal C$ as in \eqref{generaladj}. One first takes the canonical quotient $f \colon X \rightarrow X/R$, and then considers the inverse image $f^{-1} (\Eq (\eta_{X/R}))$ along $f$ of the kernel pair $\Eq(\eta_{X/R})$ of the $X/R$-component $\eta_{X/R}$ of the unit of the adjunction. Equivalently, the closure $c_X(R)$ can be defined as the kernel pair of the arrow ~$\eta_{X/R} \circ f \colon X \rightarrow UF(X/R)$ (this also shows that the equivalence relation $c_X(R)$ is effective).

\begin{examples}
	The category $\mathsf{Grp(Comp)}$ of compact (Hausdorff) groups is a regular category, and it contains as full $\mathcal E$-reflective subcategory the category $\mathsf{Grp(Prof)}$ of profinite groups. Here the $\mathcal E$-reflection of a compact group $X$ is given by the quotient $F(X) = X /\Gamma_X(0)$ by its connected component $\Gamma_X(0)$ of~$0$.
	The closure $c_X (R)$ of an effective equivalence relation $R$ on $X$ in the category of compact groups corresponding to the $\mathcal E$-reflection
	\begin{equation*}
	\vcenter{\hbox{\begin{tikzpicture}[scale=0.3]
			\node[] (X) at (-5,0) {${ \mathsf{Grp(Comp)} }$};
			\node[] (Y) at (5,0) {${ \mathsf{Grp(Prof)}}$};
			\node[] (C) at (0,0) {$\perp$};
			\draw[->] (X) to [bend left=25] node[above]{$F$} (Y);
			\draw[<-] (X) to [bend right=25] node[below]{$U$} (Y);
			\end{tikzpicture}}}
	\end{equation*}
	is given by $c_X (R)= R \circ \Gamma_X$, where two elements $x$ and $y$ in $X$ belong to the equivalence relation $\Gamma_X$ if they belong to the same connected component:
	$$\Gamma_X = \{ (x,y) \in X \times X \, \mid \, \Gamma_X(x) = \Gamma_X(y) \}.$$
	As shown in \cite{BGM} the same result still holds if the algebraic theory of groups is replaced by any Mal'tsev theory \cite{Smith} (such as the theories of rings, Lie algebras, loops, crossed modules), the proof of this essentially relying on the fact that the category of compact (Hausdorff) models of a Mal'tsev theory is a regular Mal'tsev category \cite{CLP}. \end{examples}

In what follows we shall be interested in proving a result, similar to the one recalled here above, in the case of the adjunction \eqref{adj} within the category of quandles, although this latter is not a Mal'tsev category.
The category $\Qndt$ of trivial quandles is an $\mathcal E$-reflective subcategory of the category $\Qnd$ of quandles, and this latter is a variety of universal algebras: any equivalence relation $R$ on a quandle $X$ in $\Qnd$ is then effective. An internal equivalence relation $R$ on~$X$ in $\Qnd$ is an equivalence relation on the underlying set of $X$ which is also a subquandle of the product quandle $X \times X$, i.e. a \emph{congruence} in the terminology of universal algebra \cite{BS}.
The congruence $\mathsf{Eq}(\eta_X)$ for a quandle $X$ in this case is also denoted by $\sim_{\mathsf{Inn}(X)}$ \cite{BLRY}, and is 
defined as follows: $(x,y) \in \sim_{\mathsf{Inn}(X)}$ if and only if $x$ and $y$ are in the same connected component: $[x]_X = [y]_X$. As it follows from Lemma $1.3$ in \cite{EG}, these congruences permute with any other equivalence relation $R$ on $X$ in $\mathsf{Qnd}$: 
\begin{lemma} \label{permutability}
	For any quandle $X$, the congruence $\sim_{\mathsf{Inn}(X)}$ permutes, in the sense of the composition of relations, with any congruence $R$ on $X$ in $\mathsf{Qnd}$:
\end{lemma}
\begin{equation}
\sim_{\mathsf{Inn}(X)} \circ R = R \circ \sim_{\mathsf{Inn}(X)}.
\end{equation}

The corresponding effective closure operator can be easily described thanks to Lemma~\ref{permutability}:
\begin{proposition}\label{description}
	Let $R$ be a congruence on a quandle $X$. Then its effective closure~$c_X(R)$ in $\Qnd$ corresponding to the reflection \eqref{adj} is given by $$c_X(R) =  \sim_{\mathsf{Inn}(X)} \circ R = R \circ  \sim_{\Inn(X)}.$$
\end{proposition}
\begin{proof}
	The closure~$c_X(R)$ of a congruence $R$ on a quandle $X$ is constructed as the inverse image $f^{-1}(\sim_{\Inn(X/R)})$
	of the congruence~$\sim_{\Inn(X/R)}$ along the canonical quotient $f \colon X \rightarrow X/R$:
	\[\begin{tikzpicture}[scale=0.3]
	\node[] (M) at (-4,4) {$c_X(R)$};
	\node[] (X) at (4,4) {$\sim_{\Inn(X/R)}$};
	\node[] (M') at (-4,-4) {$X$};
	\node[] (X') at (4,-4) {$X/R$};
	\node[] (Y) at (-12,-4) {$R$};
	\draw[->] (M) to (X);   
	\draw[->,transform canvas={xshift=1ex}] (M) to (M');
	\draw[->,transform canvas={xshift=-1ex}] (M) to (M');
	\draw[->,transform canvas={xshift=1ex}] (X) to node[right]{$p_2$} (X');
	\draw[->,transform canvas={xshift=-1ex}] (X) to node[left]{$p_1$} (X');
	\draw[->] (M') to node[below]{$f$} (X'); 
	\draw[->,transform canvas={yshift=1ex}] (Y) to node[above]{$f_1$} (M');
	\draw[->,transform canvas={yshift=-1ex}] (Y) to node[below]{$f_2$} (M');
	\draw[->,dotted] (Y) to (M);
	\end{tikzpicture}\]
	
	The commutative square 
	\[\begin{tikzpicture}[scale=0.35]
	\node[] (M) at (-4,4) {$X$};
	\node[] (X) at (4,4) {$X/R$};
	\node[] (M') at (-4,-4) {$U\pi_0(X)$};
	\node[] (X') at (4,-4) {$U\pi_0(X/R)$};
	\draw[->] (M) to node[above]{$f$} (X);   
	\draw[->] (M) to node[left]{$\eta_X$} (M');
	\draw[->] (X) to node[right]{$\eta_{X/R}$} (X');
	\draw[->] (M') to node[below]{$U\pi_0(f)$} (X');    
	\end{tikzpicture}\]
	induced by the units of the adjunction \eqref{adj} is a pushout, since $\Qndt$ is stable under quotients in $\Qnd$. The fact that the congruences $\sim_{\Inn(X)}$ and $R$ permute (by Lemma \ref{permutability}) implies the following equalities: $$c_X(R)= R  \vee \sim_{\Inn(X)} = R \circ \sim_{\Inn(X)},$$
	where  $R  \vee \sim_{\Inn(X)}$ denotes the supremum of $R$ and $\sim_{\Inn(X)}$ as congruences on the quandle $X$. Indeed, the supremum $R  \vee \sim_{\Inn(X)}$ is precisely the kernel pair of the composite of $U\pi_0(f) \circ \eta_X$, which certainly contains both $R$ and $\sim_{\Inn(X)}$, thus also $R \circ \sim_{\Inn(X)}$. However, since 
	$R \circ \sim_{\Inn(X)}$ is already a congruence in $\Qnd$ (by  Lemma \ref{permutability}), it is then $R  \vee \sim_{\Inn(X)}$.
\end{proof}
\begin{remark}
	Observe that, for any $X \in \Qnd$, the congruence $\Eq({\eta_X})= \sim_{\Inn(X)}$ is simply the closure of the equality relation~$\Delta_X$ on $X$: $$c_X({\Delta}_X)= \sim_{\Inn(X)}.$$
	It is not difficult to check that the effective closure operator associated with \eqref{adj} also satisfies the property that $$f(c_X(\Delta_X)) = c_Y (\Delta_Y) $$
	for any regular epimorphism $f \colon X \rightarrow Y$: this essentially follows from Corollary~$1.8$ in \cite{EG}. One can also show that, for any congruences $R$ and $S$ on the same quandle $X$, $$c_X (R \vee S) = c_X(R) \vee c_X(S).$$ Indeed, both the congruences $c_X (R \vee S)$ and $c_X(R) \vee c_X(S)$ turn out to be the kernel congruence $\mathsf{Eq}(\eta_P \circ i_1 \circ q_R)$ of the composite $\eta_P \circ i_1 \circ q_R$, where $i_1$ and $i_2$ are defined by the following pushout ($q_R$ and $q_S$ are the canonical quotients):
	\[\begin{tikzpicture}[scale=0.3]
	\node[] (M) at (-4,4) {$X$};
	\node[] (X) at (4,4) {$X/S$};
	\node[] (M') at (-4,-4) {$X/R$};
	\node[] (X') at (4,-4) {$P.$};
	\draw[->] (M) to node[above]{$q_S$} (X);   
	\draw[->] (M) to node[left]{$q_R$} (M');
	\draw[->] (X) to node[right]{$i_2$} (X');
	\draw[->] (M') to node[below]{$i_1$} (X');    
	\end{tikzpicture}\]
	
\end{remark}


\begin{thebibliography}{99}
	\bibitem{AW} A.V. Arhangel'ski\v{\i}, R. Wiegandt, \emph{Connectedness and disconnectedness in topology}, General Topology and Appl. 5 (1975) 9-33.
	\bibitem{Barr} M. Barr, P.A. Grillet, D.H. van Osdol, \emph{Exact Categories and Categories of Sheaves}, Lecture Notes in Math. 236, Springer, Berlin (1971).
	\bibitem{Borceux} F. Borceux, \emph{Handbook of Categorical Algebra 2. Categories and Structures.} Encycl. Mathem. its Applications 51, Cambridge University Press (1994).
	\bibitem{BGM} F. Borceux, M. Gran, S. Mantovani, \emph{On closure operators and reflections in Goursat categories}, Rend. Istit. Mat. Univ. Trieste 39 (2007) 87-104.
	\bibitem{BLRY} E. Bunch, P. Lofgren, A. Rapp, and D. N. Yetter, \emph{On quotients of quandles}, J. Knot Theory Ramifications 19 (2010) 9 1145-1156.
	\bibitem{BS} S. Burris and H.P. Sankappannavar, \emph{A course in universal algebra}, Graduate Texts in Mathematics  78, Springer-Verlag (1981).
	\bibitem{CLP} A. Carboni, J. Lambek, and M.C. Pedicchio,
	\emph{Diagram chasing in Mal'cev
		categories}, J. Pure Appl. Algebra 69 (1991) 271--284.
	\bibitem{CHK} C. Cassidy, M. H\'ebert, G.M. Kelly, \emph{Reflective subcategories, localizations, and factorization systems}, J. Austral. Math. Soc. Ser. A 38 (1985) 3 287-329.
	\bibitem{Clementino} M.M. Clementino, \emph{Weakly hereditary regular closure operators}, Topology  Appl. 49 (1993) 2 129-139.
	\bibitem{CT} M.M. Clementino, W. Tholen, \emph{Separated and connected maps}, Appl. Categ. Structures 6 (1998) 3 373-401.
	\bibitem{DG} D. Dikranjan, E. Giuli, \emph{Closure operators. I.}, Topology  Appl. 27 (1987) 2 129-143.
	\bibitem{DT} D. Dikranjan, W. Tholen, \emph{Categorical Structure of Closure Operators with Applications to Topology, Algebra and Discrete Mathematics}, Kluwer Academic Publishers Group, Dordrecht (1995).
	\bibitem{DT2} D. Dikranjan, W. Tholen, \emph{Dual closure operators and their applications}, J. Algebra 439 (2015) 373-416.
	\bibitem{EG} V. Even and M. Gran, \emph{On factorization systems for surjective quandle homomorphisms}, J. Knot Theory Ramifications 23 (2014) 11 1450060.
	\bibitem{Holgate} D. Holgate, \emph{The pullback closure operator and generalisations of perfectness}, Appl. Categ. Structures 4 (1996) 1 107-120.
	\bibitem{Inoue} A. Inoue, \emph{Quasi-triviality of quandles for link-homotopy}, J. Knot Theory Ramifications 22 (2013) 6 1350026.
	\bibitem{Joyce} D. Joyce, \emph{A classifying invariant of knots, the knot quandle}, J. Pure Appl. Algebra 23 (1982) 1 37-65.
	\bibitem{M} S. V. Matveev, \emph{Distributive groupoids in knot theory}, Mat. Sb. (N.S), 119(161) (1982) 1 78-88 160.
	\bibitem{S} S. Salbany, \emph{Reflective subcategories and closure operators}, Lecture Notes in Math. 540, Springer, Berlin (1976) 548-565.
	\bibitem{Smith} J.D.H. Smith, \emph{Mal'cev Varieties}, Lecture Notes in Math. 554 (1976).
\end{thebibliography}
\end{document}